\documentclass[12pt]{amsart}
\usepackage{amsmath, amsthm, amscd, amsfonts, amssymb, graphicx, color,float,pgf,tikz}
\usepackage{amssymb,fontenc}
\usepackage{latexsym,wasysym,mathrsfs}
\usepackage{hyperref}
\usetikzlibrary{arrows}
\usepackage[square,numbers,sort&compress]{natbib}

\makeatletter \oddsidemargin.9375in \evensidemargin \oddsidemargin
\newtheorem{theorem}{Theorem}[section]
\newtheorem{lemma}[theorem]{Lemma}
\newtheorem{proposition}[theorem]{Proposition}
\newtheorem{corollary}[theorem]{Corollary}

\theoremstyle{definition}
\newtheorem{definition}[theorem]{Definition}
\newtheorem{example}[theorem]{Example}

\numberwithin{equation}{section}

\newcommand{\be}{\begin{equation}}
\newcommand{\ee}{\end{equation}}
\newcommand{\bee}{\begin{example}}
\newcommand{\eee}{\end{example}}

\numberwithin{equation}{section}
\usepackage{etoolbox}
\makeatletter
\patchcmd{\@settitle}{\uppercasenonmath\@title}{}{}{}
\patchcmd{\@setauthors}{\MakeUppercase}{}{}{}
\makeatother

\allowdisplaybreaks
\begin{document}

\title[ Controlled K-operator frame for $End_\mathcal{A}^\ast (\mathcal{H})$]{Controlled K-operator frame for $End_\mathcal{A}^\ast (\mathcal{H})$}

\author[Abdeslam TOURI and Samir KABBAJ]{Abdeslam TOURI$^1$$^{*}$\MakeLowercase{and} Samir KABBAJ$^1$}

\address{$^{1}$Department of Mathematics, University of Ibn Tofail, B.P. 133, Kenitra, Morocco}
\email{\textcolor[rgb]{0.00,0.00,0.84}{touri.abdo68@gmail.com}}
\email{\textcolor[rgb]{0.00,0.00,0.84}{samkabbaj@yahoo.fr}}




\keywords{Operator Frame, Controlled operator frame, K-operator frame, Controlled K-operator frame    $C^{\ast}$-algebra, Hilbert $\mathcal{A}$-modules.\\
\indent $^{*}$ Corresponding author}
\maketitle

\begin{abstract}
	Frame Theory has a great revolution for recent years. This theory has been extended from Hilbert spaces to Hilbert  $C^{\ast}$-modules. In this paper, we introduce the concept of Controlled K-operator frame for the space $End_{\mathcal{A}}^{\ast}(\mathcal{H})$ of all adjointable operators on a Hilbert $\mathcal{A}$-module $\mathcal{H}$ and we establish some results.

\end{abstract}
\maketitle
\vspace{0.1in}

\section{\textbf{Introduction and preliminaries}}
The concept of frames in Hilbert spaces has been introduced by
Duffin and Schaeffer \cite{Duf} in 1952 to study some deep problems in nonharmonic Fourier series. After the fundamental paper \cite{13} by Daubechies, Grossman and Meyer, frame theory began to be widely used, particularly in the more specialized context of wavelet frames and Gabor frames \cite{Gab}. Frames have been used in signal processing, image processing, data compression and sampling theory. 
The concept of a generalization of frames to a family indexed by some locally compact space endowed with a Radon measure was proposed by G. Kaiser \cite{15} and independently by Ali, Antoine and Gazeau \cite{11}. These frames are known as continuous frames. Gabardo and Han in \cite{14} called these frames associated with measurable spaces, Askari-Hemmat, Dehghan and Radjabalipour in \cite{12} called them generalized frames and in mathematical physics they are referred to as coherent states \cite{11}. 
In 2012, L. Gavruta \cite{02} introduced the notion of K-frames in Hilbert space to study the atomic systems with respect to a bounded linear operator K. Controlled frames in Hilbert spaces have been introduced by P. Balazs \cite{01} to improve the numerical efficiency of iterative algorithms for inverting the frame operator. Rahimi \cite{05} defined the concept of controlled K-frames in Hilbert spaces and showed that controlled K-frames are equivalent to K-frames due to which the controlled operator C  can be used as preconditions in applications.
Controlled frames in $C^{\ast}$-modules were introduced by Rashidi and Rahimi \cite{03}, and the authors showed that they share many useful properties with their corresponding notions in a Hilbert space. K-operator frmae for $End_{\mathcal{A}}^{\ast}(\mathcal{H})$ has been study by M. Rossafi \cite{art1}.
Motivated by the above literature, we introduce the notion of a  controlled K-operator frame for Hilbert $C^{\ast}$-modules.

In the following we briefly recall the definitions and basic properties of $C^{\ast}$-algebra, Hilbert $\mathcal{A}$-modules. Our references for $C^{\ast}$-algebras as \cite{{Dav},{Con}}. For a $C^{\ast}$-algebra $\mathcal{A}$ if $a\in\mathcal{A}$ is positive we write $a\geq 0$ and $\mathcal{A}^{+}$ denotes the set of all positive elements of $\mathcal{A}$.
\begin{definition}\cite{Pas}	
	Let $ \mathcal{A} $ be a unital $C^{\ast}$-algebra and $\mathcal{H}$ be a left $ \mathcal{A} $-module, such that the linear structures of $\mathcal{A}$ and $ \mathcal{H} $ are compatible. $\mathcal{H}$ is a pre-Hilbert $\mathcal{A}$-module if $\mathcal{H}$ is equipped with an $\mathcal{A}$-valued inner product $\langle.,.\rangle_{\mathcal{A}} :\mathcal{H}\times\mathcal{H}\rightarrow\mathcal{A}$, such that is sesquilinear, positive definite and respects the module action. In the other words,
	\begin{itemize}
		\item [(i)] $ \langle x,x\rangle_{\mathcal{A}}\geq0 $ for all $ x\in\mathcal{H} $ and $ \langle x,x\rangle_{\mathcal{A}}=0$ if and only if $x=0$.
		\item [(ii)] $\langle ax+y,z\rangle_{\mathcal{A}}=a\langle x,z\rangle_{\mathcal{A}}+\langle y,z\rangle_{\mathcal{A}}$ for all $a\in\mathcal{A}$ and $x,y,z\in\mathcal{H}$.
		\item[(iii)] $ \langle x,y\rangle_{\mathcal{A}}=\langle y,x\rangle_{\mathcal{A}}^{\ast} $ for all $x,y\in\mathcal{H}$.
	\end{itemize}	 
	For $x\in\mathcal{H}, $ we define $||x||=||\langle x,x\rangle_{\mathcal{A}}||^{\frac{1}{2}}$. If $\mathcal{H}$ is complete with $||.||$, it is called a Hilbert $\mathcal{A}$-module or a Hilbert $C^{\ast}$-module over $\mathcal{A}$. For every $a$ in $C^{\ast}$-algebra $\mathcal{A}$, we have $|a|=(a^{\ast}a)^{\frac{1}{2}}$ and the $\mathcal{A}$-valued norm on $\mathcal{H}$ is defined by $|x|=\langle x, x\rangle_{\mathcal{A}}^{\frac{1}{2}}$ for $x\in\mathcal{H}$.
	
	Let $\mathcal{H}$ and $\mathcal{K}$ be two Hilbert $\mathcal{A}$-modules, A map $T:\mathcal{H}\rightarrow\mathcal{K}$ is said to be adjointable if there exists a map $T^{\ast}:\mathcal{K}\rightarrow\mathcal{H}$ such that $\langle Tx,y\rangle_{\mathcal{A}}=\langle x,T^{\ast}y\rangle_{\mathcal{A}}$ for all $x\in\mathcal{H}$ and $y\in\mathcal{K}$.
	
We reserve the notation $End_{\mathcal{A}}^{\ast}(\mathcal{H},\mathcal{K})$ for the set of all adjointable operators from $\mathcal{H}$ to $\mathcal{K}$ and $End_{\mathcal{A}}^{\ast}(\mathcal{H},\mathcal{H})$ is abbreviated to $End_{\mathcal{A}}^{\ast}(\mathcal{H})$.
\end{definition}  
The following lemmas will be used to prove our mains results.
\begin{lemma} \label{111} \cite{Ara}.
	Let $\mathcal{H}$ and $\mathcal{K}$ be two Hilbert $\mathcal{A}$-modules and $T\in End_{\mathcal{A}}^{\ast}(\mathcal{H},\mathcal{K})$. Then the following statements are equivalent:
	\begin{itemize}
		\item [(i)] $T$ is surjective.
		\item [(ii)] $T^{\ast}$ is bounded below with respect to norm, i.e, there is $m>0$ such that $\|T^{\ast}x\|\geq m\|x\|$, $x\in\mathcal{K}$.
		\item [(iii)] $T^{\ast}$ is bounded below with respect to the inner product, i.e, there is $m'>0$ such that, $$\langle T^{\ast}x,T^{\ast}x\rangle_\mathcal{A}\geq m'\langle x,x\rangle_\mathcal{A} , x\in\mathcal{K}$$.
	\end{itemize}
\end{lemma}
\begin{lemma} \label{1} \cite{Pas}.
	Let $\mathcal{H}$ be an Hilbert $\mathcal{A}$-module. If $T\in End_{\mathcal{A}}^{\ast}(\mathcal{H})$, then $$\langle Tx,Tx\rangle_{\mathcal{A}}\leq\|T\|^{2}\langle x,x\rangle_{\mathcal{A}}, \qquad x\in\mathcal{H}.$$
\end{lemma}
For the following theorem, $R(T)$ denote the range of the operator $T$.
\begin{theorem}\label{3} \cite{Do} 
	Let E, F and G be Hilbert ${\mathcal{A}}$-modules over a $C^{\ast}$-algebra $\mathcal{A}$. Let $ T\in End_{\mathcal{A}}^{\ast}(E,F) $ and  $  T'\in End_{\mathcal{A}}^{\ast}(G,F) $ with 
    $\overline{({R(T^{\ast})})}$ is orthogonally complemented. Then the following statements are equivalent:
	\begin{itemize}
		\item [(1)] $T'(T')^{\ast} \leq \lambda TT^{\ast}$ for some $\lambda >0$.
		\item  [(2)]	There exists $\mu >0 $ such that $\|(T')^{\ast}x\|\leq \mu \|T^{\ast}x\|$ for all $x\in F$.
		\item  [(3)] There exists $ D\in End^{\ast}(G,E)$ such that $T'=TD$,
		that is the equation $TX=T'$ has a solution.
		\item  [(4)]  $R(T')\subseteq R(T)$.
		
	\end{itemize}
	
\end{theorem}
\begin{definition}\label{haja14}  \cite{haj}
	Let $\mathcal{H}$ be a Hilbrt $\mathcal{A}$-module over a unital $C^{\ast}$-algebra, $C \in GL^{+}(\mathcal{H})$ and $K \in End_{\mathcal{A}}^{\ast}(\mathcal{H})$. A family $\{x_i\}_{i \in I}$ in $\mathcal{H}$ is said to be C-controlled K-frame if there exist two constants $0<A\leq B < \infty$ such that 
	$$A\langle C^{\frac{1}{2}}K^{\ast}x,C^{\frac{1}{2}}K^{\ast}x \rangle_{\mathcal{A}} \leq \sum_{i\in I} \langle x,x_i\rangle_{\mathcal{A}}  \langle Cx_i,x\rangle_{\mathcal{A}} \leq  B\langle x,x\rangle_{\mathcal{A}}, x\in \mathcal{H} $$
	
\end{definition}
\section{\textbf{Controlled K-operator frame for $End_\mathcal{A}^\ast (\mathcal{H})$}}
We begin this section with the following definitions.
\begin{definition}\cite{RR1}
	A family of adjointable operators $\{T_{i}\}_{i\in I}$ on a Hilbert $\mathcal{A}$-module $\mathcal{H}$ over a unital $C^{\ast}$-algebra is said to be an operator frame for $End_{\mathcal{A}}^{\ast}(\mathcal{H})$, if there exist two positives constants $A, B > 0$ such that 
	\begin{equation}\label{eq13}
	A\langle x,x\rangle_{\mathcal{A}}\leq\sum_{i\in I}\langle T_{i}x,T_{i}x\rangle_{\mathcal{A}}\leq B\langle x,x\rangle_{\mathcal{A}},  x\in\mathcal{H}.
	\end{equation}
	The numbers $A$ and $B$ are called lower and upper bound of the operator frame, respectively. If $A=B=\lambda$, the operator frame is called $\lambda$-tight.
	If $A = B = 1$, it is called a normalized tight operator frame or a Parseval operator frame.
	If only upper inequality of \eqref{eq13} hold, then $\{T_{i}\}_{i\in I}$ is called an operator Bessel sequence for $End_{\mathcal{A}}^{\ast}(\mathcal{H})$.\\	
	If the sum in the middle of \eqref{eq13} is convergent in norm, the operator frame is called standard.
\end{definition}
\begin{definition}\cite{art1}
	Let $K \in End_{\mathcal{A}}^{\ast}(\mathcal{H})$. A family of adjointable operators $\{T_{i}\}_{i\in I}$ on a Hilbert $\mathcal{A}$-module $\mathcal{H}$ over a unital $C^{\ast}$-algebra is said to be a  K-operator frame for $End_{\mathcal{A}}^{\ast}(\mathcal{H})$, if there exist two positives constants $A, B > 0$ such that 
	\begin{equation}\label{eq3}
	A\langle K^{\ast}x,K^{\ast}x\rangle_{\mathcal{A}}\leq\sum_{i\in I}\langle T_{i}x,T_{i}x\rangle_{\mathcal{A}}\leq B\langle x,x\rangle_{\mathcal{A}},  x\in\mathcal{H}.
	\end{equation}
	The numbers $A$ and $B$ are called lower and upper bound of the operator frame, respectively. If $A=B=\lambda$, the operator frame is called $\lambda$-tight.
	If $A = B = 1$, it is called a normalized tight operator frame or a Parseval operator frame.
	If only upper inequality of \eqref{eq3} hold, then $\{T_{i}\}_{i\in I}$ is called an operator Bessel sequence for $End_{\mathcal{A}}^{\ast}(\mathcal{H})$.\\	
	If the sum in the middle of \eqref{eq3} is convergent in norm, the K-operator frame is called standard.
\end{definition}
    Let $GL^{+}(\mathcal{H})$ be the set for all positive bounded linear invertible operators on $\mathcal{H}$ with bounded inverse.
\begin{definition}\cite{lab}
	Let $C,C^{'} \in GL^{+}(\mathcal{H})$, a family of adjointable operators $\{T_{i}\}_{i\in I}$ on a Hilbert $\mathcal{A}$-module $\mathcal{H}$ over a unital $C^{\ast}$-algebra is said to be a $(C,C^{'})$-controlled operator frame for $End_{\mathcal{A}}^{\ast}(\mathcal{H})$, if there exist two positive constants $A , B >0$ such that
	\begin{equation}\label{t1}
	A\langle x,x\rangle_{\mathcal{A}} \leq\sum_{i\in I}\langle T_{i}Cx,T_{i}C^{'}x\rangle_{\mathcal{A}}\leq B\langle x,x\rangle_{\mathcal{A}},  x\in\mathcal{H}.
	\end{equation}
	The numbers $A$ and $B$ are called lower and upper bounds of the $(C,C^{'})$-controlled operator frame , respectively.\\
	If $A=B=\lambda$, the $(C,C^{'})$-controlled operator frame  is called $\lambda$-tight.\\
	If $A = B = 1$, it is called a normalized tight $(C,C^{'})$-controlled operator frame or a Parseval $(C,C^{'})$-controlled operator frame .\\
	If only upper inequality of \eqref{t1} hold, then $\{T_{i}\}_{i\in i}$ is called a $(C,C^{'})$-controlled operator Bessel sequence for $End_{\mathcal{A}}^{\ast}(\mathcal{H})$.
\end{definition}
\begin{definition}
	Let $K \in End_{\mathcal{A}}^{\ast}(\mathcal{H})$ and  $C,C^{'} \in GL^{+}(\mathcal{H})$. A family of adjointable operators $\{T_{i}\}_{i\in I}$ on a Hilbert $\mathcal{A}$-module $\mathcal{H}$ over a unital $C^{\ast}$-algebra is said to be a $(C,C^{'})$-controlled K-operator frame for $End_{\mathcal{A}}^{\ast}(\mathcal{H})$, if there exist two positives constants $A , B >0$ such that
	\begin{equation}\label{1}
	A\langle K^{\ast}x,K^{\ast}x\rangle_{\mathcal{A}} \leq\sum_{i\in I}\langle T_{i}Cx,T_{i}C^{'}x\rangle_{\mathcal{A}}\leq B\langle x,x\rangle_{\mathcal{A}},  x\in\mathcal{H}.
	\end{equation}
	The numbers $A$ and $B$ are called lower and upper bounds of the $(C,C^{'})$-controlled K-operator frame , respectively.\\
	If $A=B=\lambda$, the $(C,C^{'})$-controlled K-operator frame  is called $\lambda$-tight.\\
	If $A = B = 1$, it is called a normalized tight $(C,C^{'})$-controlled K-operator frame or a Parseval $(C,C^{'})$-controlled operator frame .\\
	If only upper inequality of \eqref{1} hold, then $\{T_{i}\}_{i\in i}$ is called an $(C,C^{'})$-controlled K-operator Bessel sequence for $End_{\mathcal{A}}^{\ast}(\mathcal{H})$.
\end{definition}
\begin{example}
	Let $C \in GL^{+}(\mathcal{H})$, $K \in End_{\mathcal{A}}^{\ast}(\mathcal{H})$ and $\{x_{i}\}_{i\in I}$ be a $C$-controlled K-frame for $\mathcal{H}$.\\
	Let $(\Gamma_{i})_{i\in I} \in End_{\mathcal{A}}^{\ast}(\mathcal{H})$ such that :
	\begin{equation*}
\Gamma_{i}(x)=\langle x,x_{i}\rangle e_{i}, \qquad for\; all \quad i\in I \quad and \quad x\in \mathcal{H},
	\end{equation*}
	 where $\langle e_i,e_j\rangle_\mathcal{A} =\delta_{ij} 1_{\mathcal{A}}$.\\
	 Since $\{x_{i}\}_{i\in I}$ is a $C$-controlled K-frame for $\mathcal{H}$, then by
	 definition \ref{haja14} , there exist $A,B > 0$ such that :
	\begin{equation*}
	A\langle C^{\frac{1}{2}} K^{\ast} x,C^{\frac{1}{2}}K^{\ast}x\rangle_{\mathcal{A}} \leq \sum_{i\in I} \langle x,x_{i}\rangle_{\mathcal{A}} \langle C x_{i},x\rangle_{\mathcal{A}} \leq  B \langle x,x\rangle_{\mathcal{A}} \qquad x\in \mathcal{H}
	\end{equation*}
	Hence,
	\begin{equation*}
	A\langle C^{\frac{1}{2}} K^{\ast} x,C^{\frac{1}{2}}K^{\ast}x\rangle_{\mathcal{A}} \leq \sum_{i\in I} \langle x,x_{i}\rangle_{\mathcal{A}} \langle e_i,e_i\rangle_{\mathcal{A}} \langle C x_{i},x\rangle_{\mathcal{A}} \leq  B \langle x,x\rangle_{\mathcal{A}} \qquad x\in \mathcal{H}.
	\end{equation*}
	So,
	$$A\langle C^{\frac{1}{2}} K^{\ast} x,C^{\frac{1}{2}}K^{\ast}x\rangle_{\mathcal{A}} \leq \sum_{i\in I}  \langle\langle x,x_{i}\rangle_{\mathcal{A}} e_i,\langle x ,Cx_{i}\rangle_{\mathcal{A}}e_i\rangle_{\mathcal{A}}  \leq  B \langle x,x\rangle_{\mathcal{A}} \qquad x\in \mathcal{H}.$$
	Since C is a selfadjoint operator then,
	\begin{equation*}
	A\langle C^{\frac{1}{2}} K^{\ast} x,C^{\frac{1}{2}}K^{\ast}x\rangle_{\mathcal{A}} \leq \sum_{i\in I}  \langle\langle x,x_{i}\rangle_{\mathcal{A}} e_i,\langle Cx ,x_{i}\rangle_{\mathcal{A}}e_i\rangle_{\mathcal{A}}  \leq  B \langle x,x\rangle_{\mathcal{A}} \qquad x\in \mathcal{H}.
	\end{equation*}
	Therefore,
	\begin{equation*}
	A\langle C^{\frac{1}{2}} K^{\ast} x,C^{\frac{1}{2}}K^{\ast}x\rangle_{\mathcal{A}} \leq \sum_{i\in I}  \langle\Gamma_i x,\Gamma_i Cx\rangle_{\mathcal{A}}  \leq  B \langle x,x\rangle_{\mathcal{A}} \qquad x\in \mathcal{H}.
	\end{equation*}
	Since $C$ is a surjective operator, from lemma \ref{111}, there exists $m>0$, such that
	\begin{equation*}
m \langle K^{\ast} x,K^{\ast}x\rangle_{\mathcal{A}}\leq \langle C^{\frac{1}{2}} K^{\ast} x,C^{\frac{1}{2}}K^{\ast}x\rangle_{\mathcal{A}}.
	\end{equation*}
Then, 
\begin{equation*}
Am \langle K^{\ast} x,K^{\ast}x\rangle_{\mathcal{A}} \leq \sum_{i\in I}  \langle\Gamma_i x,\Gamma_i Cx\rangle_{\mathcal{A}}  \leq  B \langle x,x\rangle_{\mathcal{A}} \qquad x\in \mathcal{H}.
\end{equation*}
	Then $(\Gamma_{i})_{i\in I}$ is a $(Id_{\mathcal{H}},C)$-controlled $K$-operator frame for $End_{\mathcal{A}}^{\ast}(\mathcal{H})$.
\end{example}
\begin{proposition}
	Every $(C,C^{'})$-controlled operator frame for $End_{\mathcal{A}}^{\ast}(\mathcal{H})$ is a $(C,C^{'})$-controlled K-operator frame for $End_{\mathcal{A}}^{\ast}(\mathcal{H})$.
\end{proposition}
\begin{proof}
	For any $K \in End_{\mathcal{A}}^{\ast}(\mathcal{H}) $, we have,
	$$\langle K^{\ast}x,K^{\ast}x\rangle_{\mathcal{A}} \leq \|K\|^2 \langle x,x\rangle_{\mathcal{A}}.$$
	Let $\{T_{i}\}_{i\in I}$ be a $(C,C^{'})$-controlled operator frame for $End_{\mathcal{A}}^{\ast}(\mathcal{H})$ with bounds A and B.\\
	Then, 
	$$A\langle x,x\rangle_{\mathcal{A}} \leq\sum_{i\in I}\langle T_{i}Cx,T_{i}C^{'}x\rangle_{\mathcal{A}}\leq B\langle x,x\rangle_{\mathcal{A}},  x\in\mathcal{H}.$$
	Hence, 
	$$A\|K\|^{-2} \langle K^{\ast}x,K^{\ast}x\rangle_{\mathcal{A}} \leq\sum_{i\in I}\langle T_{i}Cx,T_{i}C^{'}x\rangle_{\mathcal{A}}\leq B\langle x,x\rangle_{\mathcal{A}},  x\in\mathcal{H}.$$
	Therefore, $\{T_{i}\}_{i\in I}$  is a $(C,C^{'})$-controlled K-operator frame for $End_{\mathcal{A}}^{\ast}(\mathcal{H})$ with bounds $A\|K\|^{-2}$ and B.
	
\end{proof}

\begin{proposition}
	Let $\{T_{i}\}_{i\in I}$ be a $(C,C^{'})$-controlled K-operator frame for $End_{\mathcal{A}}^{\ast}(\mathcal{H})$. If K is surjective then $\{T_{i}\}_{i\in I}$ is a $(C,C^{'})$-controlled operator frame for $End_{\mathcal{A}}^{\ast}(\mathcal{H})$.

\end{proposition}

\begin{proof}
	Suppose that K is surjetive, from lemma \ref{111}  there exists $0<m$ such that
	
\begin{equation} \label{do11}
      \langle K^{\ast}x,K^{\ast}x\rangle_\mathcal{A}\geq m\langle x,x\rangle_\mathcal{A} , x\in\mathcal{H}
\end{equation}
   Let $\{T_{i}\}_{i\in I}$ be a $(C,C^{'})$-controlled K-operator frame for $End_{\mathcal{A}}^{\ast}(\mathcal{H})$ with bounds A and B.
    Hence, 
\begin{equation} \label{do12}
    A\langle K^{\ast}x,K^{\ast}x\rangle_{\mathcal{A}} \leq\sum_{i\in I}\langle T_{i}Cx,T_{i}C^{'}x\rangle_{\mathcal{A}}\leq B\langle x,x\rangle_{\mathcal{A}},  x\in\mathcal{H}.
\end{equation}
    Using \ref{do11}	and \ref{do12}, we have
    $$A m\langle x,x\rangle_{\mathcal{A}} \leq\sum_{i\in I}\langle T_{i}Cx,T_{i}C^{'}x\rangle_{\mathcal{A}}\leq B\langle x,x\rangle_{\mathcal{A}},  x\in\mathcal{H}.$$
    Therefore  $\{T_{i}\}_{i\in I}$ is a $(C,C^{'})$-controlled operator frame for $End_{\mathcal{A}}^{\ast}(\mathcal{H})$.
\end{proof}
\begin{proposition}
	Let  $C,C^{'} \in GL^{+}(\mathcal{H})$ and $\{T_{i}\}_{i\in I}$ be a  K-operator frame for $End_{\mathcal{A}}^{\ast}(\mathcal{H})$. Assume that $C$ and $C'$ commute with $T_i$ and K.  Then $\{T_{i}\}_{i\in I}$ is a $(C,C^{'})$-controlled K-operator frame for $End_{\mathcal{A}}^{\ast}(\mathcal{H})$.
	
\end{proposition}
\begin{proof}
	Let $\{T_{i}\}_{i\in I}$ be a  K-operator frame for $End_{\mathcal{A}}^{\ast}(\mathcal{H})$.\\
	Then there exist $A, B > 0$ such that 
	\begin{equation}\label{eq33}
	A\langle K^{\ast}x,K^{\ast}x\rangle_{\mathcal{A}}\leq\sum_{i\in I}\langle T_{i}x,T_{i}x\rangle_{\mathcal{A}}\leq B\langle x,x\rangle_{\mathcal{A}},  x\in\mathcal{H}.
	\end{equation} 
	Since, 
\begin{align*}
    \sum_{i\in I}\langle T_{i}Cx,T_{i}C'x\rangle_{\mathcal{A}}&= \sum_{i\in I}\langle T_{i}(CC')^{\frac{1}{2}}x,T_{i}(CC')^{\frac{1}{2}}x\rangle_{\mathcal{A}}\\ 
    &\leq B\langle (CC')^{\frac{1}{2}} x,(CC')^{\frac{1}{2}} x\rangle_{\mathcal{A}}\\
    &\leq B \|(CC')^{\frac{1}{2}}\|^2 \langle x,x\rangle_{\mathcal{A}},
\end{align*}
   then, 
\begin{equation} \label{do23}
     \sum_{i\in I}\langle T_{i}Cx,T_{i}C'x\rangle_{\mathcal{A}}\leq B \|(CC')^{\frac{1}{2}}\|^2 \langle x,x\rangle_{\mathcal{A}}.
\end{equation}
    Moreover, 
\begin{align*}
    \sum_{i\in I}\langle T_{i}(CC')^{\frac{1}{2}}x,T_{i}(CC')^{\frac{1}{2}}x\rangle_{\mathcal{A}}
    & \geq 	A\langle K^{\ast}(CC')^{\frac{1}{2}}x,K^{\ast}(CC')^{\frac{1}{2}}x\rangle_{\mathcal{A}}\\ 
    &\geq 	A\langle (CC')^{\frac{1}{2}} K^{\ast}x,(CC')^{\frac{1}{2}}K^{\ast}x\rangle_{\mathcal{A}}.
\end{align*} 
    Since $(CC')^{\frac{1}{2}}$ is a surjecive operator, then there exists $m>0$ such that,
\begin{equation} \label{do24}
    \langle (CC')^{\frac{1}{2}} K^{\ast}x,(CC')^{\frac{1}{2}}K^{\ast}x\rangle_{\mathcal{A}} \geq m \langle  K^{\ast}x,K^{\ast}x\rangle_{\mathcal{A}}.
\end{equation} 
    From \ref{do23} and \ref{do24}, we have,
    $$ A m \langle  K^{\ast}x,K^{\ast}x\rangle_{\mathcal{A}} \leq \sum_{i\in I}\langle T_{i}Cx,T_{i}C'x\rangle_{\mathcal{A}}\leq B \|(CC')^{\frac{1}{2}}\|^2 \langle x,x\rangle_{\mathcal{A}}. $$ 
    Therefore $\{T_{i}\}_{i\in I}$ is a $(C,C^{'})$-controlled K-operator frame for $End_{\mathcal{A}}^{\ast}(\mathcal{H})$. 
	
\end{proof}
 Let $\{T_i\}_{i\in I}$ be a $(C,C')-$controlled Bessel K-operator frame for $End_\mathcal{A}^\ast (\mathcal{H})$. We assume that $C$ and $C'$ commute with $T_i$ and $T_i^{\ast}$.\\
    We define the operator $T_{(C,C')}$ :$ \mathcal{H} \rightarrow l^2(\mathcal{H})$
    by $$T_{(C,C')}x=\{T_i(CC')^\frac{1}{2}x\}_{i\in I}.$$
    There adjoint operator is defined by $T^{\ast}_{(C,C^{'})} : l^2(\mathcal{H})  \rightarrow \mathcal{H}$ given by, $$T^{\ast}_{(C,C^{'})}(\{a_i\}_{i \in I})=\sum_{i\in I} (CC')^\frac{1}{2}T_i^{\ast}a_i$$
    is called the synthesis operator.\\
    If $C$ and $C^{'}$ commute between them, and commute with the operators $T_i^{\ast}T_i$ for each $i \in I$. We define the $(C,C^{'})$-controlled Bessel K-operator frame  by:
 \begin{align*}
    S_{(C,C^{'})}:&\mathcal{H}\longrightarrow \mathcal{H}\\
    &x\longrightarrow S_{(C,C^{'})}x=T_{(C,C^{'})}T^{\ast}_{(C,C^{'})}x=\sum_{i\in I}  C'T_i^{\ast}T_iCx.
\end{align*}
     
    $T_{(C,C')}$ and  $T_{(C,C')}^\ast$ are called the synthesis and analysis operator of $(C,C')-$controlled Bessel K-operator frame $\{T_{i}\}_{i\in I}$ respectively.\\
    
    It's clear to see that $S_{(C,C^{'})}$ is positive, bounded and selfadjoint.
   
\begin{theorem}
	 Let $\{T_i\}_{i\in I}$ be a $(C,C')-$controlled Bessel K-operator frame for $End_\mathcal{A}^\ast (\mathcal{H})$. The following statements are equivalent:
\begin{itemize}
	\item [(1)] $\{T_i\}_{i\in I}$ is a $(C,C')-$controlled  K-operator frame.
	\item [(2)] There is $A>0$ such that $S_{(C,C')}\geq AKK^{\ast}$.
	\item [(3)] $K=S_{(C,C')}^\frac{1}{2}Q$, for some $Q \in End_\mathcal{A}^\ast (\mathcal{H})$.
\end{itemize}
	
\end{theorem}
\begin{proof}
	$(1) \Longrightarrow (2)$\\
	Assume that $\{T_i\}_{i\in I}$ is a $(C,C')-$controlled  K-operator frame for $End_\mathcal{A}^\ast (\mathcal{H})$ with bounds A and B, with frame operator $S_{(C,C')}$, then
	$$A\langle K^{\ast}x,K^{\ast}x\rangle_{\mathcal{A}} \leq\sum_{i\in I}\langle T_{i}Cx,T_{i}C^{'}x\rangle_{\mathcal{A}}\leq B\langle x,x\rangle_{\mathcal{A}},  x\in\mathcal{H}.$$
    Therefore, 
    $$A\langle K K^{\ast}x,x\rangle_{\mathcal{A}} \leq\langle \sum_{i\in I} C'T_i^{\ast} T_{i}Cx,x\rangle_{\mathcal{A}}\leq B\langle x,x\rangle_{\mathcal{A}},  x\in\mathcal{H}.$$
    Hence, 
    $$S_{(C,C')}\geq AKK^{\ast}.$$
   	$(2) \Longrightarrow (3)$\\
   	Let $A>0$ such that $$S_{(C,C')}\geq AKK^{\ast}.$$
   	This give,
   	$$S_{(C,C')}^{\frac{1}{2}}S_{(C,C')}^{{\frac{1}{2}} \ast}\geq AKK^{\ast}.$$

   	From theorem \ref{3}, we have,
   	$$K=S_{(C,C')}^{\frac{1}{2}}Q$$
   	with $Q \in End_\mathcal{A}^\ast (\mathcal{H})$.\\
   	$(3) \Longrightarrow (1)$\\
   	Suppose that , $$K=S_{(C,C')}^{\frac{1}{2}}Q$$ for some $Q \in End_\mathcal{A}^\ast (\mathcal{H})$.\\
   	From theorem \ref{3} there exists $A>0$ such that, 
   	$$AKK^{\ast} \leq S_{(C,C')}^{\frac{1}{2}}S_{(C,C')}^{{\frac{1}{2}} \ast}.$$
   	Hence, $$ AKK^{\ast} \leq S_{(C,C')}.$$
   	Therefore,  $\{T_i\}_{i\in I}$ is a $(C,C')-$controlled  K-operator frame for $End_\mathcal{A}^\ast (\mathcal{H})$.\\
   	\end{proof}   	
\begin{theorem}
	Let $K,Q \in End_\mathcal{A}^\ast (\mathcal{H})$ and $\{T_i\}_{i\in I}$ be a $(C,C')-$controlled  K-operator frame for $End_\mathcal{A}^\ast (\mathcal{H})$.
	Suppose that $Q$ commute with $C$ , $C'$ and $K$. Then $\{T_iQ\}_{i\in I}$ is a $(C,C')-$controlled  $Q^{\ast}K$-operator frame for $End_\mathcal{A}^\ast (\mathcal{H})$.
	
\end{theorem}
\begin{proof}
	Suppose that $\{T_i\}_{i\in I}$  is a $(C,C')-$controlled  K-operator frame with frame bounds A and B. Then, $$A\langle K^{\ast}x,K^{\ast}x\rangle_{\mathcal{A}} \leq\sum_{i\in I}\langle T_{i}Cx,T_{i}C^{'}x\rangle_{\mathcal{A}}\leq B\langle x,x\rangle_{\mathcal{A}},  x\in\mathcal{H}.$$
	Hence, 
	$$A\langle K^{\ast}Qx,K^{\ast}Qx\rangle_{\mathcal{A}} \leq\sum_{i\in I}\langle T_{i}CQx,T_{i}C^{'}Qx\rangle_{\mathcal{A}}\leq B\langle Qx,Qx\rangle_{\mathcal{A}},  x\in\mathcal{H}.$$
	So, 
	$$A\langle (Q^{\ast}K)^{\ast}x,(Q^{\ast}K)^{\ast}x\rangle_{\mathcal{A}} \leq\sum_{i\in I}\langle T_{i}QCx,T_{i}QC^{'}x\rangle_{\mathcal{A}}\leq B \|Q\|^2\langle x,x\rangle_{\mathcal{A}},  x\in\mathcal{H}.$$
	Therefore $\{T_iQ\}_{i\in I}$ is a $(C,C')$-controlled  $Q^{\ast}K$-operator frame for $End_\mathcal{A}^\ast (\mathcal{H})$ with bounds A and $B \|Q\|^2$.
	
\end{proof}
\begin{theorem}
	Let $K \in End_\mathcal{A}^\ast (\mathcal{H})$ and $\{T_i\}_{i\in I}$ be a $(C,C')$-controlled tight K-operator frame for $End_\mathcal{A}^\ast (\mathcal{H})$ with bound $A_1$.
	Then  $\{T_i\}_{i\in I}$  is a $(C,C')$-controlled tight operator frame for $End_\mathcal{A}^\ast (\mathcal{H})$
	with bound $A_2$ if and only if $K_r^{-1}=\frac{A_1}{A_2} K^\ast$.
	
\end{theorem}
\begin{proof}
	Let $\{T_i\}_{i\in I}$ be a $(C,C')-$controlled tight K-operator frame for $End_\mathcal{A}^\ast (\mathcal{H})$ with bound $A_1$.\\
	Assume that $\{T_i\}_{i\in I}$ is a $(C,C')-$controlled tight operator frame for $End_\mathcal{A}^\ast (\mathcal{H})$ with bound $A_2$. Then, 
	$$\sum_{i\in I}\langle T_{i}Cx,T_{i}C^{'}x\rangle_{\mathcal{A}}=A_2\langle x,x\rangle_{\mathcal{A}},  x\in\mathcal{H}.$$
	Since $\{T_i\}_{i\in I}$ is a $(C,C')-$controlled tight K-operator frame for $End_\mathcal{A}^\ast (\mathcal{H})$, then we have, 
	$$A_1\langle K^{\ast}x,K^{\ast}x  \rangle_{\mathcal{A}} =\sum_{i\in I}\langle T_{i}Cx,T_{i}C^{'}x\rangle_{\mathcal{A}}.$$
	Hence, 
	$$A_1\langle K^{\ast}x,K^{\ast}x  \rangle_{\mathcal{A}} = A_2\langle x,x\rangle_{\mathcal{A}}.$$
	So, 
    $$\langle KK^{\ast}x,x  \rangle_{\mathcal{A}} =\langle \frac{A_2}{A_1} x,x\rangle_{\mathcal{A}},  x\in\mathcal{H}.$$
    Then, 
    $$KK^{\ast} =  \frac{A_2}{A_1}Id_{\mathcal{A}}.$$
    Therefore, $$K_r^{-1}=\frac{A_1}{A_2} K^{\ast}.$$
    For the converse, assume that  $$K_r^{-1}=\frac{A_1}{A_2} K^{\ast}.$$
    Then, $$KK^{\ast} =  \frac{A_2}{A_1}Id_{\mathcal{A}}.$$ 
    This give that, $$\langle KK^{\ast}x,x  \rangle =\langle \frac{A_2}{A_1} x,x\rangle_{\mathcal{A}},  x\in\mathcal{H}.$$
    Since $\{T_i\}_{i\in I}$ is a $(C,C')-$controlled tight K-operator frame for $End_\mathcal{A}^\ast (\mathcal{H})$ with bound $A_1$, the we have,
   $$\sum_{i\in I}\langle T_{i}Cx,T_{i}C^{'}x\rangle_{\mathcal{A}}=A_2\langle x,x\rangle_{\mathcal{A}},  x\in\mathcal{H}.$$ 
   Therefore  $\{T_i\}_{i\in I}$ is a $(C,C')-$controlled tight operator frame for $End_\mathcal{A}^\ast (\mathcal{H})$. 
\end{proof}
\begin{corollary}
	Let $K \in End_\mathcal{A}^\ast (\mathcal{H})$ and  $\{T_i\}_{i\in I}$ be a sequence for $End_\mathcal{A}^\ast (\mathcal{H})$. Then those statements are true,\\
	(1) If $\{T_i\}_{i\in I}$ is a $(C,C')-$controlled  tight K-operator frame for $End_\mathcal{A}^\ast (\mathcal{H})$, then $\{T_i(K^n)^{\ast}\}_{i\in I}$ is a $(C,C')-$controlled  tight $K^{n+1}$-operator  frame for $End_\mathcal{A}^\ast (\mathcal{H})$ .\\
	(2) If $\{T_i\}_{i\in I}$ is a $(C,C')-$controlled  tight operator frame for $End_\mathcal{A}^\ast (\mathcal{H})$ then $\{T_iK^{\ast}\}_{i\in I}$ is a $(C,C')-$controlled  tight K-operator frame for $End_\mathcal{A}^\ast (\mathcal{H})$.
\end{corollary}
\begin{theorem}
	Let $\{T_i\}_{i\in I}$ be a $(C,C')-$controlled  K-operator frame for $End_\mathcal{A}^\ast (\mathcal{H})$ with best frame bounds A and B. If  $Q : \mathcal{H}  \rightarrow \mathcal{H}$
	is an adjointable and invertible operator such that $Q^{-1}$ commutes with $K^{\ast}$, then $\{T_iQ\}_{i\in I}$ is a $(C,C')-$controlled  K-operator frame for $End_\mathcal{A}^\ast (\mathcal{H})$ with best frame bounds M and N satisfying the inequalities,

\begin{equation}\label{eq10}
A\|Q^{-1}\|^{-2}\leq M \leq A\|Q\|^2 \;\;\;\;\;\; and \;\;\;\;\;\; A\|Q^{-1}\|^{-2}\leq N \leq B\|Q\|^2.
\end{equation}

\end{theorem}
\begin{proof}
	Let $\{T_i\}_{i\in I}$ be a $(C,C')-$controlled  K-operator frame for $End_\mathcal{A}^\ast (\mathcal{H})$ with best frame bounds A and B.
	Then, $$\sum_{i\in I}\langle T_{i}CQx,T_{i}C^{'}Qx\rangle_{\mathcal{A}}\leq B\langle Qx,Qx\rangle_{\mathcal{A}}\leq B \|Q\|^2\langle x,x\rangle_{\mathcal{A}}. $$
	Also we have,
\begin{align*}
    	A\langle K^{\ast}x,K^{\ast}x\rangle_{\mathcal{A}}&=A\langle K^{\ast}Q^{-1}Qx,K^{\ast}Q^{-1}Qx\rangle_{\mathcal{A}}\\
    	&=A\langle Q^{-1} K^{\ast}Qx,Q^{-1} K^{\ast}Qx\rangle_{\mathcal{A}}\\
    	&\leq \| Q^{-1}\|^2 \sum_{i\in I}\langle T_{i}CQx,T_{i}C^{'}Qx\rangle_{\mathcal{A}}\\
    	&= \| Q^{-1}\|^2 \sum_{i\in I}\langle T_{i}QCx,T_{i}QC^{'}x\rangle_{\mathcal{A}}.
\end{align*}
       Hence, 
       $$A\| Q^{-1}\|^{-2}\langle K^{\ast}x,K^{\ast}x\rangle_{\mathcal{A}}\leq \sum_{i\in I}\langle T_{i}QCx,T_{i}QC^{'}x\rangle_{\mathcal{A}}\leq B \|Q\|^2\langle x,x\rangle_{\mathcal{A}}.$$
       Therefore, $\{T_iQ\}_{i\in I}$ is a $(C,C')-$controlled  K-operator frame for $End_\mathcal{A}^\ast (\mathcal{H})$ with bounds $A\| Q^{-1}\|^{-2}$ and $ B \|Q\|^2$.\\
       Now let M and N be the best bounds of the $(C,C')-$controlled  K-operator frame $\{T_iQ\}_{i\in I}$. Then,
       
\begin{equation}\label{eq8}
       A\|Q^{-1}\|^{-2} \leq M\;\;\;\;\;\; and \;\;\;\;\;\;N\leq B \|Q\|^2 .
\end{equation}
    
       Also, $\{T_{i}Q\}_{i\in I}$ is a $(C,C')-$controlled $K$-operator frame for $End_{\mathcal{A}}^{\ast}(\mathcal{H})$ with frame bounds $M$ and $N$, and
       $$\langle K^{\ast}x,K^{\ast}x\rangle_{\mathcal{A}} =\langle QQ^{-1}K^{\ast}x,QQ^{-1}K^{\ast}x\rangle_{\mathcal{A}} \leq\|Q\|^{2}\langle K^{\ast}Q^{-1}x,K^{\ast}Q^{-1}x\rangle_{\mathcal{A}} , x\in\mathcal{H}.$$
      Hence
      \begin{align*}
      M\|Q\|^{-2}\langle K^{\ast}x,K^{\ast}x\rangle_{\mathcal{A}} &\leq M\langle K^{\ast}Q^{-1}x,K^{\ast}Q^{-1}x\rangle_{\mathcal{A}} \\
      &\leq\sum_{i\in\mathbb{J}}\langle T_{i}QCQ^{-1}x,T_{i}QC'Q^{-1}x\rangle_{\mathcal{A}} \\
      &\leq\sum_{i\in\mathbb{J}}\langle T_{i}QQ^{-1}Cx,T_{i}QQ^{-1}C'x\rangle_{\mathcal{A}} \\
      &=\sum_{i\in I}\langle T_{i}Cx,T_{i}C^{'}x\rangle_{\mathcal{A}} \\
      &\leq N\|Q^{-1}\|^{2}\langle x,x\rangle_{\mathcal{A}} .
      \end{align*}
      Since A and B are the best bounds of $(C,C')-$controlled  K-operator frame $\{T_i\}_{i\in I}$, we have 
      \begin{equation}\label{eq9}
      C\|Q\|^{-2}\leq A \;\;\;\;\;\; and \;\;\;\;\;\; B\leq D\|Q^{-1}\|^{2}.
      \end{equation}
      Therfore the inequality \ref{eq10} follows from \ref{eq9}   and \ref{eq8}.
\end{proof}
\bibliographystyle{amsplain}

\vspace{0.1in}

\end{document}